%% file: arxiv.tex
\documentclass[onecolumn, a4paper, 12pt, conference]{ieeeconf}
\IEEEoverridecommandlockouts                              
\let\labelindent\relax

\usepackage[utf8]{inputenc}
\usepackage{graphics} 
\usepackage{amsmath} 
\usepackage{amssymb}  
\usepackage{amsfonts}
\usepackage{comment}
\usepackage{enumitem}
\usepackage{cases,xspace,enumitem}
\usepackage{mdframed}
\usepackage{stmaryrd}

\usepackage{subcaption}
\usepackage{multirow}
\usepackage[rightcaption]{sidecap}

\DeclareCaptionLabelFormat{figure}{{Figure #2}}
\usepackage{tikz}
\usetikzlibrary{positioning}

\usepackage[prependcaption,colorinlistoftodos]{todonotes}

\definecolor{gnred}{RGB}{255,91,89}

\definecolor{gnred1}{RGB}{71,0,0} 
\definecolor{gnred2}{RGB}{117,0,0} 
\definecolor{gnred3}{RGB}{164,0,0} 
\definecolor{gnred4}{RGB}{211,0,0} 
\definecolor{gnred5}{RGB}{255,0,0} 
\definecolor{gnred6}{RGB}{255,42,34} 
\definecolor{gnred7}{RGB}{255,91,89} 

\definecolor{gnblue1}{RGB}{0,36,71}   
\definecolor{gnblue2}{RGB}{0,60,118}  
\definecolor{gnblue3}{RGB}{0,85,164}
\definecolor{gnblue4}{RGB}{0,108,212}
\definecolor{gnblue4}{RGB}{0,108,212}
\definecolor{gnblue5}{RGB}{0,133,255}  
\definecolor{gnblue6}{RGB}{35,156,255} 
\definecolor{gnblue7}{RGB}{88,177,255} 

\definecolor{gnbrown1}{RGB}{71,27,0}  
\definecolor{gnbrown2}{RGB}{117,45,0} 
\definecolor{gnbrown3}{RGB}{164,62,0} 
\definecolor{gnbrown4}{RGB}{211,80,0} 
\definecolor{gnbrown5}{RGB}{255,97,0} 
\definecolor{gnbrown6}{RGB}{255,127,26} 
\definecolor{gnbrown7}{RGB}{255,155,86} 

\renewcommand{\theenumi}{(\roman{enumi})}

\newcommand\Item[1][]{%
  \ifx\relax#1\relax  \item \else \item[#1] \fi
  \abovedisplayskip=0pt\abovedisplayshortskip=0pt~\vspace*{-\baselineskip}}

\usepackage{hyperref}
\hypersetup{
    colorlinks=true,
    linkcolor=blue,
    filecolor=magenta,      
    urlcolor=cyan,
    pdftitle={Overleaf Example},
    pdfpagemode=FullScreen,
    }

\usepackage{amsthm}

\newtheoremstyle{ieeeconf}
  {0pt}   
  {0pt}   
  {\normalfont}  
  {\parindent}       
  {\itshape} 
  {:}         
  { } 
  {\thmname{#1} \thmnumber{#2}\thmnote{ (#3)}} 
\makeatletter
\renewenvironment{proof}[1][\proofname]{\par
  \pushQED{\qed}%
  \normalfont \topsep\z@
  \trivlist
  \item[\hskip2em
        \itshape
    #1\@addpunct{:}]\ignorespaces
}{%
  \popQED\endtrivlist\@endpefalse
}
\makeatletter

\theoremstyle{ieeeconf}

\input{my_commands.tex}
\newcommand\oprocendsymbol{\hbox{$\triangle$}}
\newcommand\oprocend{\relax\ifmmode\else\unskip\hfill\fi\oprocendsymbol}

\usepackage{titlesec}

\titlespacing*{\section}
{0pt}{3ex plus 1ex minus .2ex}{2ex plus .2ex}
\titlespacing*{\subsection}
{0pt}{3ex plus 1ex minus .2ex}{2ex plus .2ex}
\titlespacing*{\subsubsection}
{0pt}{1.5ex plus 1ex minus .2ex}{1ex plus .2ex}
\titlespacing*{\paragraph}
{0pt}{1.5ex plus 1ex minus .2ex}{1ex plus .2ex}

\title{On Weakly Contracting Dynamics for Convex Optimization}

\author{Veronica Centorrino$^a$, Alexander Davydov$^b$, Anand Gokhale$^b$,\\
Giovanni Russo$^c$ and Francesco Bullo$^b$
\thanks{$^a$Veronica Centorrino is with Scuola Superiore Meridionale, Italy. {\tt\small veronica.centorrino@unina.it.}}
\thanks{$^b$Alexander Davydov, Anand Gokhale, and Francesco Bullo are with the Center for Control, Dynamical 
Systems, and Computation, UC Santa Barbara, Santa Barbara, CA 93106 USA. {\tt\small 
davydov@ucsb.edu, anand\_gokhale@ucsb.edu, bullo@ucsb.edu}.}
\thanks{$^c$Giovanni Russo is with the Department of Information and Electric Engineering and Applied Mathematics, University of Salerno, Italy. {\tt\small giovarusso@unisa.it.}}
}

\date{}
\begin{document}
\pagestyle{plain}

\maketitle

\begin{abstract}
\normalsize
We analyze the convergence behavior of \emph{globally weakly} and \emph{locally strongly contracting} dynamics.
Such dynamics naturally arise in the context of convex optimization problems with a unique minimizer. 
We show that convergence to the equilibrium is \emph{linear-exponential}, in the sense that the distance between each solution and the equilibrium is upper bounded by a function that first decreases linearly and then exponentially.  
As we show, the linear-exponential dependency arises naturally in certain dynamics with saturations. 
Additionally, we provide a sufficient condition for local input-to-state stability. 
Finally, we illustrate our results on, and propose a conjecture for, continuous-time dynamical systems solving linear programs.\end{abstract}
\section{Introduction}\label{2024a-sec:introduction}
\textit{Problem description and motivation: }A paradigm that is becoming popular to analyze possibly time-varying optimization problems (OP) is to synthesize continuous-time dynamical systems that {\em converge} to an equilibrium that is also the optimal solution of the problem. A suitable tool to assess convergence is contraction theory~\cite{WL-JJES:98, FB:23-CTDS}. 
For OPs with strongly convex costs, the corresponding gradient dynamics, primal-dual dynamics (in the presence of constraints), or proximal gradient dynamics (for non-smooth costs) are strongly contracting, implying that trajectories exponentially converge to the equilibrium, which is also the optimal solution. In contrast, for OPs with only convex costs, the corresponding gradient, primal-dual, or proximal gradient dynamics are weakly contracting (or nonexpansive), and convergence depends on the existence of the minimizer.

In this context, we focus on convex OPs with a unique minimizer via continuous-time dynamical systems. These OPs lead to a class of continuous-time dynamical systems that are globally-weakly contracting in the state space and only locally-strongly contracting (\gwlsc). We characterize the convergence behavior of such dynamics, showing that this is \emph{linear-exponential}, and local input-to-state stability (ISS).
Finally, as an application, we consider linear programming (LP) to illustrate the effectiveness of our results.

\textit{Literature review: }
Studying optimization algorithms as continuous-time dynamical systems has been an active research area since~\cite{KJA-LH-HU:58}, with, e.g.,~\cite{DWT-JJH:84} being one of the first works to design neural networks for LPs. ~\cite{EKPC-SH-SHZ:99} proposed a neural network  based on non-differentiable penalty functions for solving LPs. Recent advancements in, e.g., online and dynamic feedback optimization~\cite{GB-JC-JIP-EDA:22} have renewed the interest in continuous-time dynamics for optimization.
Additionally, OPs have been related to dynamical systems via proximal gradients, and the corresponding continuous-time proximal gradient dynamics are studied in, e.g.,~\cite{NKD-SZK-MRJ:19, SHM-MRJ:21}. Proximal gradients dynamics and contraction theory have been exploited in~\cite{AD-VC-AG-GR-FB:23f},~\cite{VC-AG-AD-GR-FB:23a} for tackling problems with strongly convex and only convex costs, respectively.
In a broader context, there has been a growing interest in using strongly contracting dynamics to tackle OPs~\cite{HDN-TLV-KT-JJES:18,AD-VC-AG-GR-FB:23f}. This is mainly due to the fact that such dynamics enjoy highly ordered transient and asymptotic behaviors. 
Specifically, (i) initial conditions are exponentially forgotten and the distance between any two trajectories decays exponentially quickly~\cite{WL-JJES:98}, (ii) unique globally exponential stable equilibrium for time-invariant dynamics~\cite{WL-JJES:98} (iii) entrainment to periodic inputs~\cite{GR-MDB-EDS:10a} (iv) highly robust behavior, such as ISS~\cite{SX-GR-RHM:21}. The asymptotic behavior of weakly contracting dynamics is instead characterized in, e.g.,~\cite{SC:19} for monotone systems and in~\cite{SJ-PCV-FB:19q} for primal-dynamics with a locally stable equilibrium.

\textit{Contributions:} We analyze convergence of the class of \gwlsc dynamics, showing that this is {\em linear-exponential}, in the sense that the distance between each solution of the system and the equilibrium is upper bounded by a {\em linear-exponential function}, introduced in this letter. Through a novel technical result, we characterize the evolution of certain dynamics with saturation in terms of the linear-exponential function. This lemma is exploited for our convergence analysis, which is carried out considering two cases that require distinct mathematical approaches. First, we consider systems that are \gwlsc with respect to (w.r.t.) the same norm. Then, we consider the case where the dynamics is \gwlsc w.r.t. two different norms. Specifically, we give a convergence bound that, as discussed below, is sharper than the one in~\cite{VC-AG-AD-GR-FB:23a}. Additionally, we characterize local ISS for input-dependent dynamics that are \gwlsc w.r.t. the same norm. Finally, we show the effectiveness of our results by considering a continuous-time dynamics tackling LPs and propose a general conjecture. The code to replicate our numerical example is given at \url{https://shorturl.at/vGNY1}.

While the treatment in this paper is inspired by the results in~\cite[Section 5]{VC-AG-AD-GR-FB:23a} we extend the results of \cite{VC-AG-AD-GR-FB:23a} in several ways. First, the linear-exponential function that bounds the convergence behavior of GW-LS-C dynamics is introduced in this paper. We also show in a novel lemma the relationship of this function with a scalar saturated ODE. Second, when the dynamics is \gwlsc w.r.t. two different norms, the bound we give here is continuous and always sharper than the one given in~\cite{VC-AG-AD-GR-FB:23a}, see Remark~\ref{rem:comparison} for the details. Third, when the dynamics is \gwlsc w.r.t. the same norm, the technical lemma used to establish linear-exponential convergence is novel and requires a different mathematical treatment than that of~\cite{VC-AG-AD-GR-FB:23a}.
Finally, in this paper we characterize local ISS. This property was not considered in~\cite{VC-AG-AD-GR-FB:23a}.

\section{Mathematical Preliminaries}\label{2024-sec:math_preliminaries}
We denote by $\0_n \in \R^n$ the all-zeros vector of size $n$. Vector inequalities of the form $x \leq (\geq) y$ are entrywise. We let $I_n$ be the $n \times n$ identity matrix.
Given $A, B \in \R^{n\times n}$ symmetric, we write $A \preceq B$ (resp. $A \prec B$) if $B-A$ is positive semidefinite (resp. definite). We denote by $\subscr{\lambda}{max}(A)$ the maximum eigenvalue of $A$.
We say that $A$ is \emph{Hurwitz} if $\alpha(A) := \max \setdef{\operatorname{Re}(\lambda)}{\lambda \text{ eigenvalue of } A}<0$, where $\operatorname{Re}(\lambda)$ denotes the real part of $\lambda$.

\paragraph{Norms, Logarithmic Norms and Weak Pairings}
We let $\| \cdot \|$ denote both a norm on $\R^n$ and its corresponding induced matrix norm on $\R^{n \times n}$. Given $x \in \R^n$ and $\radius >0$, we let $\ball{p}{x,\radius}:= \setdef{z\in\R^n}{\norm{z-x}_p\leq\radius}$ be the \emph{ball of radius $\radius$ centered at $x$} computed with respect to the norm $p$.
Given two norms $\norm{\cdot}_{\alpha}$ and $\norm{\cdot}_{\beta}$ on $\R^n$ there exist positive \emph{equivalence coefficients} $k_\alpha^\beta$ and $k_\beta^\alpha$ satisfying $\norm{x}_{\alpha}\leq\kab\norm{x}_{\beta}$, $\norm{x}_{\beta}\leq\kba\norm{x}_{\alpha}$, $\forall x\in\R^n$. The \emph{equivalence ratio between $\norm{\cdot}_{\alpha}$ and $\norm{\cdot}_{\beta}$} is $\prodeqnorm_{\alpha,\beta} :=\kab\kba$, with $\kab$ and $\kba$ minimal equivalence coefficients.

\smallskip
Given $A \in \R^{n \times n}$ the \emph{logarithmic norm} (log-norm) induced by $\| \cdot \|$ is $$\mu(A) := \lim_{h\to 0^+} \frac{\norm{I_n{+}h A} -1}{h}.$$
For an $\ell_p$ norm, $p \in [1,\infty]$, and for an invertible $Q \in \R^{n\times n}$, the $Q$-weighted $\ell_p$ norm is $\norm{x}_{p,Q} := \norm{Q x}_p$. The corresponding log-norm is $\wlognorm{p}{Q}{A} =\lognorm{p}{Q A Q^{-1}}$.

We let $\WP{\cdot}{\cdot}$ denote a weak pairing on $\R^n$ compatible with the norm $\norm{\cdot}$.
We recall some of the main standing assumption on weak paring useful for our analysis.
\bd
A weak pairing $\WP{\cdot}{\cdot}$, compatible with the norm $\norm{\cdot}$, satisfies:
\begin{enumerate}
\item \label{prop:wp:sub-additivity}
\emph{sub-additivity of first argument:} $\WP{x+z}{y} \leq \WP{x}{y} + \WP{z}{y}$, \ for all $x,y,z \in \R^n$;
\item \label{prop:wp:curve_derivative}
\emph{curve norm derivative formula:} $\norm{y(t)}D^+ \norm{y(t)} = \WP{\dot y(t)}{y(t)}$, for every differentiable curve $\map{y}{]a, b[}{\R^n}$ and for almost every $t \in ]a,b]$;
\item\label{prop:wp:cauchy-schwartz}
\emph{Cauchy-Schwartz inequality:} $\abs{\WP{x}{y}} \leq \norm{x}\norm{y}$, \ for all $x, y \in \R^n$;
\item \label{prop:wp:lumers_eq}
\emph{Lumer’s equality:} $\ds \mu\bigl(A\bigr)= \sup_{z \in \R^n, z \neq \0_n} \frac{\WP{Az}{z}}{\WP{z}{z}}$, for every $A\in \R^{n\times n}$.
\end{enumerate}
\ed
We refer to~\cite{FB:23-CTDS} for a recent review of those tools.
\paragraph{Mathematical Operators}
Given two normed spaces $(\mcX, \|\cdot\|_{\mcX})$, $(\mcY,\|\cdot\|_{\mcY})$, a map $\map{T}{\mcX}{\mcY}$ is Lipschitz with constant $L \geq 0$ if $\|T(x_1)-T(x_2)\|_{\mcY} \leq L\|x_1-x_2\|_{\mcX}$, for all $x_1,x_2 \in \mcX$.
The \emph{upper-right Dini derivative} of a function $\map{f}{\R}{\R}$ is $D^+f := \limsup_{h\to 0^+} \bigl(f(t+h) - f(t)\bigr)/h$.
The \emph{ceiling function}, $\map{\ceil{ \ }}{\R}{\Z}$, is defined by $\ceil{x} = \min\setdef{y \in \Z}{x \leq y}$. Given $d > 0$, the \emph{saturation function}, $\map{\mathrm{sat}_d}{\R}{[-d,d]}$, is defined by $\sat{d}{x} = x$ if $\abs{x} \leq d$, $\sat{d}{x} = d$ if $x > d$, and $\sat{d}{x} = - d$ if $x < - d$.
Given a set $\mcC$, the function $\map{\iota_{\mcC}}{\R^n}{[0,+\infty]}$ is the \emph{zero-infinity indicator function on $\mcC$} and is defined by $\iota_{\mcC}(x) = 0$ if $x \in \mcC$ and $\iota_{\mcC}(x) = +\infty$ otherwise. 
The \emph{indicator function on $\mcC$}, $\map{\indicator{\mcC}}{\R}{\{0, 1\}}$, is defined by $\indicator{\mcC}(x) = 1$ if $x \in \mcC$ and $\indicator{\mcC}(x) = 0$ otherwise.
The function $\map{\relu}{\R}{\R_{\geq 0}}$, is defined by $\relu(x) = \max\{0, x\}$.

We recall the following~\cite{FHC:83, AD-AVP-FB:22q}:
\bt[Mean value theorem for locally Lipschitz function]
\label{thm:mean_value_thm}
Let $\mcC \subseteq \R^n$ be open and convex, $\map{f}{\mcC}{\R^m}$ locally Lipschitz. Then, for almost every $x,y \in \mcC$ it holds:
$$
f(x)-f(y)=\left(\int _{0}^{1}Df(y + s(x-y))ds\,\right)(x-y),
$$
where the integral of a matrix is to be understood component wise.
\et
Whenever it is clear from the context, we omit to specify the dependence of functions on time $t$.

\subsection{Proximal Operator}
Given $\map{g}{\R^n}{\realextended}:= [{-}\infty,+\infty]$, the \emph{epigraph} of $g$ is the set $\operatorname{epi}(g) = \setdef{(x,y) \in \R^{n+1}}{g(x) \leq y}$. The map $g$ is (i) \emph{convex} if its epigraph is a convex set, (ii) \emph{proper} if its value is never $-\infty$ and is finite somewhere, and (iii) \emph{closed} if it is proper and its epigraph is a closed set.

The \emph{proximal operator of $g$ with parameter $\gamma > 0$}, $\map{\prox{\gamma g}}{\R^n}{\R^n}$, is defined by
\beq
\prox{\gamma g}(x) = \argmin_{z \in \R^n} g(z) + \frac{1}{2\gamma}\|x - z\|_2^2,
\eeq
the associated \emph{Moreau envelope}, $\map{M_{\gamma g}}{\R^n}{\R}$, and its gradient are given by:
\begin{align}
M_{\gamma g}(x) &= g(\prox{\gamma g}(x)) + \frac{1}{2\gamma} \|x - \prox{\gamma g}(x)\|_2^2, \\
\nabla M_{\gamma g}(x) &= \frac{1}{\gamma}(x - \prox{\gamma g}(x)). \label{eq:Moreau-gradient}
\end{align}
The gradient of the Moreau envelope always exists and is Lipschitz on $(\R^n, \|\cdot\|_2)$ with constant $1/\gamma$.

Finally, we recall that given a convex set $\mcC$, the proximal operator of the zero-infinity indicator function on $\mcC$ is the Euclidean projection onto $\mcC$, that is $\ds \prox{\gamma \iota_{\mcC}}(x) = \proj_{\mcC}(x) := \argmin_{z \in \mcC}\norm{x - z}_2 \in \mcC$.
\subsection{Contraction Theory for Dynamical Systems}
Consider a dynamical system 
\beq
\label{eq:dynamical_system}
\dot{x}(t) = f\bigl(t,x(t)\bigr),
\eeq 
where $\map{f}{\R_{\geq 0} \times \mcC}{\R^n}$, is a smooth nonlinear function with $\mcC\subseteq \R^n$ forward invariant set for the dynamics. We let $t \mapsto \odeflowtx{t}{x_0}$ be the flow map of~\eqref{eq:dynamical_system} at time $t$ starting from initial condition $x(0):= x_0$.
Then, we give the following:
\bd[Contracting systems] \label{def:contracting_system}
Given a norm $\norm{\cdot}$ with associated log-norm $\mu$, a smooth function $\map{f}{\R_{\geq 0} \times \mcC}{\R^n}$, with $\mcC \subseteq \R^n$ $f$-invariant, open and convex, and a constant $c >0$ ($c = 0)$ referred as \emph{contraction rate}, $f$ is $c$-strongly (weakly) infinitesimally contracting on $\mcC$ if
\beq\label{cond:contraction_log_norm}
\mu\bigl(Df(t, x)\bigr) \leq -c,  \textup{ for all } x \in \mcC  \textup{ and } t\in \R_{\geq0},
\eeq
where $Df(t,x) := \partial f(t,x)/\partial x$.
\ed
\smallskip
If $f$ is contracting, then for any two trajectories $x(\cdot)$ and $y(\cdot)$ of~\eqref{eq:dynamical_system} it holds that
$$\|\odeflowtx{t}{x_0} - \phi_t(y_0)\| \leq \e^{-ct}\|x_0 -y_0\|, \quad \textup{ for all } t \geq 0,$$
i.e., the distance between the two trajectories converges exponentially with rate $c$ if $f$ is $c$-strongly infinitesimally contracting, and never increases if $f$ is weakly infinitesimally contracting.

In~\cite[Theorem 16]{AD-AVP-FB:22q} condition~\eqref{cond:contraction_log_norm} is generalized for locally Lipschitz function, for which, by Rademacher’s theorem, the Jacobian exists almost everywhere (a.e.) in $\mcC$. Specifically, if $f$ is locally Lipschitz, then  $f$ is infinitesimally contracting on $\mcC$ if condition~\eqref{cond:contraction_log_norm} holds for almost every $x \in \mcC$ and $t\in \R_{\geq0}$.

\section{Linear-Exponential Decay of Globally-Weakly and Locally-Strongly Contracting Systems}
\label{2024a-sec:lin_exp_convergence}
In this section, we conduct a comprehensive analysis of the convergence of \gwlsc systems. First, we define the \emph{linear-exponential function}, which plays a pivotal role in bounding the convergence behavior of such dynamics.
\bd[Linear-exponential function]
\label{def:lin_exp_function}
Given a \emph{linear decay rate} $\cl>0$, an \emph{intercept} $q>0$, an \emph{exponential decay rate} $\ce>0$, and a \emph{linear-exponential crossing time} $\tld < {q}/{\cl}$, the \emph{linear-exponential function} $\map{\linexp{\cdot}}{\R_{\geq 0}}{\R_{\geq 0}}$ is defined by
\beq
\label{eq:lin_exp_function}
\linexp{t} =
\begin{cases}
q - \cl t  \ & \textup{ if } t \leq \tld,\\
\bigl(q - \cl \tld\bigr) \e^{ - \ce (t - \tld)}\  & \textup{ if }  t > \tld.\\
\end{cases}
\eeq
We write $\linexp{t \; ;\; q, \cl, \ce, \tld}$ when we want to highlight the parameters in~\eqref{eq:lin_exp_function}. See Figure~\ref{fig:lin_exp} for an illustration of~\eqref{def:lin_exp_function} for some parameters.
\ed
\begin{SCfigure}[5][!ht]
\centering
\includegraphics[width=.63\linewidth]{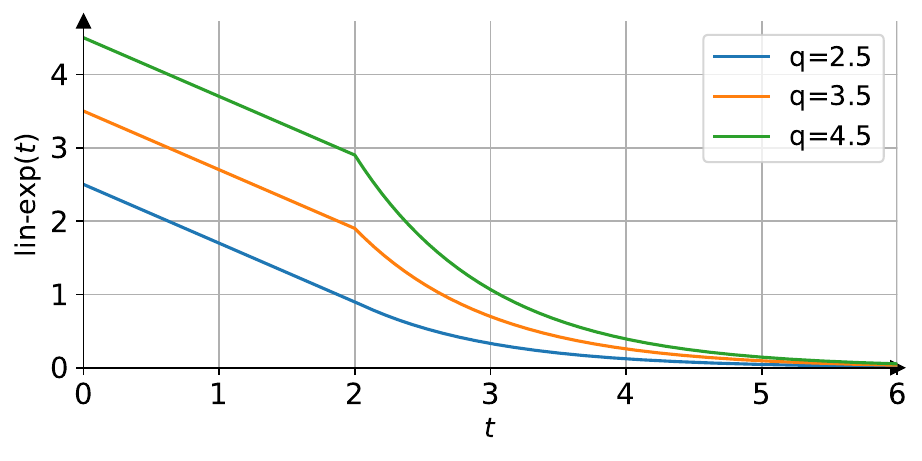}
\caption{Plot of the linear-exponential function~\eqref{eq:lin_exp_function} with linear decay rate $\cl = 0.8$, intercepts $q = \{2.5,\, 3.5,\, 4.5\}$, exponential decay rate $\ce = 1$, and linear-exponential crossing time $\tld = 2$.}
\label{fig:lin_exp}
\end{SCfigure}

Before giving the main convergence results of this section, we prove the following:
\begin{lem}[Property of the linear-exponential function]
\label{lem:linear-decay}
Let $\ce$ and $d$ be positive scalars. Consider the dynamics
\beq
\label{eq:LED:ode}
\dot x(t) =  - \ce \sat{d}{x(t)}, \quad x_0 = q > d.
\eeq
Then, $x(t) = \linexp{t \; ;\; q, \cl, \ce, \tld}$, with $\cl = d \ce$ and $\tld := q\cl^{-1} - \ce^{-1}>0$, is a solution of \eqref{eq:LED:ode}.
\end{lem}
\begin{proof}
First, we note that being the right end side of~\eqref{eq:LED:ode} locally Lipschitz continuous, the ODE~\eqref{eq:LED:ode} admits a unique continuous solution at least within a certain neighborhood of the initial condition.

Using the definition of saturation function, for all $t \in \R_{\geq 0}$ we can write the ODE~\eqref{eq:LED:ode} as
\beq
\label{eq:LED:ode_proof}
\dot x(t)  =
\begin{cases}
d \ce & \textup{ if } x(t) < -d,\\
- \ce x(t) & \textup{ if } x(t) \in [-d, d],\\
- d \ce & \textup{ if } x(t) > d,\\
\end{cases}
\eeq
which, in each interval $[t_0, t_1] \subseteq \R_{\geq 0}$ where the solution is continuous and does not change regime, has general solution
\beq
\label{eq:LED:ode_proof_sol}
x(t) =
\begin{cases}
d \ce t + x(t_0) & \textup{ if } x(t) < -d,\\
x(t_0)e^{- \ce  (t - t_0)} & \textup{ if } x(t) \in [-d, d],\\
- d \ce t + x(t_0) & \textup{ if } x(t) > d.\\
\end{cases}
\eeq
At time $t = 0$, we have $x_0 = q >d$. For continuity of the solution, there exists $\tstar$ such that $x(t) > d$ for all $t \in [0,\tstar]$. Thus from~\eqref{eq:LED:ode_proof_sol} and being $x(t_0 = 0) = q$, it is $x(t) = - d \ce t + q $, for all $t \in [0,\tstar]$. Moreover being $x(t)$ a decreasing function, the time value $\tstar$ is finite and there exists a time, say it $\bar t$ such that $x(\bar t) = d$. Let $\cl : = d \ce$, we have
\[
x(\bar t) = d \iff  - d \ce \bar t + q = d  \iff \bar t = q\cl^{-1} - \ce^{-1} := \tld.
\]

In summary, we have shown that the solution of~\eqref{eq:LED:ode_proof} is equal to $x(t) = q - \cl t $ for all $t \in [0,\tld]$ and is equal to $d$ at time $\tld$. Therefore from~\eqref{eq:LED:ode_proof_sol} and being $x(\tld) = q - \cl \tld$, for all time $t > \tld$ we have $x(t) = \bigl(q - \cl \tld\bigr) e^{- \ce  (t - \tld)}$. Specifically, $x(t) >0$ for all $t > \tld$, thus it can never be the case $x(t) < -d$. This concludes the proof.
\end{proof}

\smallskip
Next, we study the convergence behavior of \gwlsc systems of the form of~\eqref{eq:dynamical_system}, where the function $\map{f}{\R_{\geq 0} \times \mcC}{\R^n}$ is locally Lipschitz and with $\mcC\subseteq \R^n$ being $f$-invariant, open and convex. In what follows, we make the following:
\begin{assumptionbox}
\label{ass:block_assumption}
There exist $\|\cdot\|_\glo, \|\cdot\|_\loc$ on $\R^n$ such that
\begin{enumerate}[label=\textup{($A$\arabic*)}, leftmargin=1.2 cm,noitemsep]
\item \label{ass:global-weak_contractivity}
$f$ is weakly infinitesimally contracting on $\R^n$ w.r.t. $\norm{\cdot}_\glo$;
\item \label{ass:local-strong_contractivity}
$f$ is $\ce$-strongly infinitesimally contracting on a forward-invariant set $\mcS$ w.r.t. $\norm{\cdot}_\loc$;
\item \label{ass:eq_point}
$\xstar \in \mcS$ is an equilibrium point, i.e., $f(t,\xstar) = \0_n$, for all $t \geq 0$.
\end{enumerate}
\end{assumptionbox}

\begin{rem}
Assumptions~\ref{ass:local-strong_contractivity},~\ref{ass:eq_point} can be equivalently replaced by assuming the existence of a locally exponentially stable equilibrium.
\end{rem}
First, we consider \gwlsc systems with respect to the same norm. Then, dynamics that are \gwlsc with respect to different norms.
In both scenarios, we show that the convergence is (globally) \emph{linear-exponential}. That is, given a trajectory $x(t)$ of the dynamics, the distance $\norm{x(t) - \xstar}_\glo$ is upper bounded by a linear-exponential function~\eqref{eq:lin_exp_function}.

\subsection{Convergence of Globally-Weakly and Locally-Strongly Contracting Dynamics with Respect to the Same Norm}
We start by giving a bound on the upper right Dini derivative of the distance of any solution of~\eqref{eq:dynamical_system} with respect to the equilibrium $\xstar$.
\begin{lem}[Saturated error dynamics]
\label{lem:saturated_error_dynamics}
Consider system~\eqref{eq:dynamical_system} and let Assumptions~\ref{ass:global-weak_contractivity} --~\ref{ass:eq_point} hold with $\norm{\cdot}_\glo = \norm{\cdot}_\loc := \norm{\cdot}$. Let $\radius$ be the largest radius such that $\ball{}{\xstar, r} \subseteq \mcS$. Then, for every trajectory $x(t)$ starting from $x_0 \notin \mcS$, for almost every $t \geq 0$, we have
\begin{align}
\label{ineq:LED-weaklycontract+LES}
D^+ \norm{x(t)-\xstar}  &\leq -\ce \satur_\radius (\norm{x(t)-\xstar}).
\end{align}
\end{lem}
\begin{proof}
Consider an arbitrary trajectory $x(t)$ starting from $x_0 \notin \mcS$ and a second trajectory equal to the equilibrium $\xstar$. Let $\mu$ be the log-norm associated to $\norm{\cdot}$. For for almost every $t \geq 0$ it holds (\cite{FHC:83, AD-AVP-FB:22q}):
\begin{equation*}
D^+ \norm{x(t)-\xstar} \leq \int_0^1 \mu\Bigl(
D{f}\bigl(t,\xstar+\alpha(x(t)-\xstar)\bigr)\Bigr) d\alpha \,\cdot\, \norm{x(t)-\xstar} := RHS.
\end{equation*}
where $\alpha \in [0,1]$, and $\xstar+\alpha(x(t)-\xstar)$ is the segment from $\xstar$ to $x(t)$.

For each $t \geq 0$, if $\norm{x(t) - \xstar} \leq \radius$, then Assumption~\ref{ass:local-strong_contractivity} implies
\begin{equation*}
RHS \leq \int_0^1 (-\ce) d\alpha \,\cdot\, \norm{x(t)-\xstar} = - \ce \norm{x(t)-\xstar} = - \ce \sat{\radius}{\norm{x(t)-\xstar}},
\end{equation*}
where in the last equality we have used the definition of saturation function.
If $\norm{x(t)-\xstar}\geq\radius$, define $\alpha^*=\radius/\norm{x(t)-\xstar}$ and note that,  for almost every $t \geq 0$, Assumptions~\ref{ass:local-strong_contractivity} and~\ref{ass:global-weak_contractivity} imply
\begin{align*}
\alpha\leq\alpha^* \quad & \implies\quad \mu \bigl(D{f}(t,\xstar+\alpha(x(t)-\xstar))\bigr)\leq-\ce, \\
\alpha>\alpha^* \quad&\implies\quad \mu(D{f}(t,\xstar+\alpha(x(t)-\xstar))\bigr)\leq 0.
\end{align*}
Therefore, for almost every $t \geq 0$, it holds
\begin{align}
RHS &\leq \Bigl( \int_0^{\alpha^*} \!\!
\mu\bigl(D{f}(t,\xstar+\alpha(x(t)-\xstar))\bigr) d\alpha +
\int_{\alpha^*}^1 \!\!
\mu\bigl(D{f}\bigl(t,\xstar+\alpha(x(t)-\xstar))\bigr)d\alpha \Bigr) \,\cdot\,
\norm{x(t)-\xstar} \nonumber \\
&\leq (-\ce \alpha^* + 0 )\norm{x(t)-\xstar} = -\ce\radius = - \ce \sat{\radius}{\norm{x(t)-\xstar}},
\label{ineq:convergence-weak-new}
\end{align}
where in the last equality we used the definition of saturation function.
Figure~\ref{fig:average_lognorm_same_norms} provides an illustration of this result about the average of the log norm. This concludes the proof.
\begin{SCfigure}[7][!h]
\centering
\includegraphics[width=.56\linewidth]{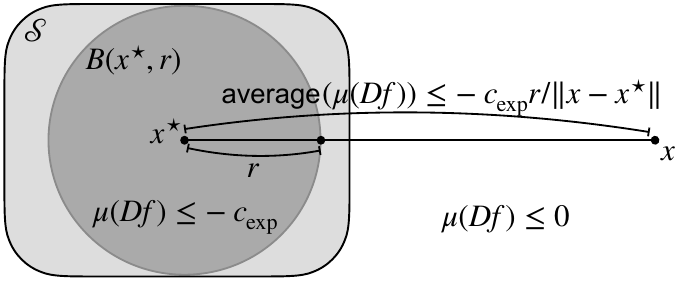}
\caption{Illustration of the inequality~\eqref{ineq:convergence-weak-new} with $\norm{\cdot} = \norm{\cdot}_{2}$.}
\label{fig:average_lognorm_same_norms}
\end{SCfigure}
\end{proof}
\smallskip
We can now give our convergence result for \gwlsc systems with respect to the same norm.

\smallskip
\bt[Linear-exponential convergence of \gwlsc systems w.r.t. the same norm]
\label{thm:local-exp-contractivity_same_norms}
Consider system~\eqref{eq:dynamical_system} and let Assumptions~\ref{ass:global-weak_contractivity} --~\ref{ass:eq_point} hold with $\norm{\cdot}_\glo = \norm{\cdot}_\loc := \norm{\cdot}$.
Also, let $\radius$ be the largest radius such that $\ball{}{\xstar, r} \subseteq \mcS$. For each trajectory $x(t)$ starting from $x_0$, it holds that
\begin{enumerate}
\item
\label{item:1_equal_norm}
if $x_0 \in \mcS$, then, for almost every $t \geq 0$,
\[
\norm{x(t) - \xstar} \leq \e^{-\ce t} \norm{x_0 - \xstar};
\]
\item
\label{item:2_equal_norm}
if $x_0 \notin \mcS$, then, for almost every $t \geq 0$,
\beq
\label{eq:bound_outside_the_ball}
\norm{x(t) - \xstar} \leq \linexp{t \; ;\; q, \cl, \ce, \tld},
\eeq
with
\bi
\item exponential decay rate $\ce > 0$;
\item linear decay rate $\cl = \ce\,\radius$;
\item intercept $\ds q = \norm{x_0-\xstar}$;
\item linear-exponential crossing time $\tld ={(q - \radius)}/{\cl}$.
\ei
\end{enumerate}
\et
\begin{proof}
Item~\ref{item:1_equal_norm} follows from Assumption~\ref{ass:local-strong_contractivity}.
Item~\ref{item:2_equal_norm} follows by using the Comparison Lemma~\cite[pp.~102-103]{HKK:02} to upper bound the solution to the differential inequality~\eqref{ineq:LED-weaklycontract+LES}.
Additionally, the upper bound obeys precisely the initial value~\eqref{eq:LED:ode} in Lemma~\ref{lem:linear-decay}, for parameter values $d=\radius$, $\cl = \ce \radius$, $q = \norm{x_0-\xstar}$, and $\tld ={(q - \radius)}/{\cl}$.
This concludes the proof.
\end{proof}

\subsection{Convergence of Globally-Weakly and Locally-Strongly Contracting Dynamics with Respect to Different Norms}
\label{2024a-sec:diff_norms}
We begin by introducing the \emph{$\rho$-contraction time}, where $0<\rho<1$.
\bd[$\rho$-contraction time]
\label{def:contraction_time}
Let system~\eqref{eq:dynamical_system} be strongly infinitesimally contracting with respect to a norm $\norm{\cdot}_{\alpha}$. Consider the \emph{contraction factor} $0<\rho<1$, a norm $\norm{\cdot}_{\beta}$, and a vector $x \in \R^n$.
\bi
\item The \emph{$\rho$-contraction time} is the time required for each trajectory starting in $\ball{\alpha}{x, r}$, for some $r>0$, to be inside $\ball{\alpha}{x, \rho r}$;
\item The \emph{$\rho$-contraction time with respect to $\norm{\cdot}_{\beta}$} is the time required for each trajectory starting in $\ball{\beta}{x, r}$, for some $r>0$, to be inside $\ball{\beta}{x, \rho r}$.
\ei
\ed
\smallskip
\begin{rem}
It is implicit in Definition~\ref{def:contraction_time} that the $\rho$-contraction time for a specific trajectory depends on the initial condition and the center of the ball.
\end{rem}
\smallskip
We can now give our convergence result for \gwlsc systems with respect to the different norms.
\bt[Linear-exponential convergence of \gwlsc systems]
\label{thm:lin_exp_decay_GWC_and_LSC_dynamics_diff_norms}
Let $\norm{\cdot}_{\loc}$ and $\norm{\cdot}_{\glo}$ be two norms on $\R^n$ with equivalence ratio $\prodeqnorm_{\loc,\glo}$. Consider system~\eqref{eq:dynamical_system} satisfying Assumptions~\ref{ass:global-weak_contractivity} -- \ref{ass:eq_point}.
Let $\radius$ be the largest radius such that $\ball{\glo}{\xstar, r} \subseteq \mcS$. For each trajectory $x(t)$ starting from $x_0$, it holds that
\begin{enumerate}
\item
if $x_0 \in \mcS$, then, for almost every $t \geq 0$,
\beq
\label{eq:bound_inside_the_ball}
\norm{x(t) - \xstar}_\glo \leq \prodeqnorm_{\loc, \glo}\e^{-\ce t} \norm{x_0 - \xstar}_\glo;
\eeq
\label{item:1_diff_norm}
\item if $x_0 \notin \mcS$, then for any \emph{contractor factor} $0 < \rho < 1$ and, for almost every $t \geq 0$,
\begin{align}
\label{eq:bound_outside_the_ball_diff_norms}
\norm{x(t) - \xstar}_{\glo}&\leq \linexp{t \; ;\; q, \cl, \ce, \tld},
\end{align}
with
\bi
\item exponential decay rate $\ce > 0$;
\item linear decay rate $\cl = \ce \radius (1-\rho)/\ln(\prodeqnorm_{\loc,\glo} \rho^{-1})$;
\item intercept $\ds q = \norm{x_0 - \xstar}_\glo + \radius(1 - \rho)\frac{\ln(\ratiolg)}{\ln(\ratiolg \rho^{-1})}$;
\item linear-exponential crossing time $\tld = \Bigceil{\frac{\glbnorm{x_0-\xstar} - \radius}{(1-\rho)\radius}}{\ln(\ratiolg \rho^{-1})}/{\ce} + \ln(\ratiolg) / \ce$.
\ei
\label{item:2_diff_norm}
\end{enumerate}
\et
\begin{proof}
Consider a trajectory $x(t)$ starting from initial condition $x_0$. If $x_0\in\mcS$, then item~\ref{item:1_diff_norm} follows from assumption~\ref{ass:local-strong_contractivity} and the equivalence of norms.

Indeed, Assumption~\ref{ass:local-strong_contractivity} implies that for every $x_0 \in \mcS$ and for almost every $t \geq 0$, it holds
\[
\norm{\odeflowtx{t}{x_0} - \xstar}_{\loc} \leq \e^{-\ce t} \norm{x_0 - \xstar}_{\loc}.
\]
Applying the equivalence of norms to the above inequality, we get
\begin{align}
\norm{\odeflowtx{t}{x_0} - \xstar}_{\glo} \leq \prodeqnorm_{\loc, \glo} \e^{-\ce t} \norm{x_0 - \xstar}_{\glo}.
\label{eq:condition_lsc_glo_diff_norms}
\end{align}

If $x_0\notin\mcS$, define the point $y_0 := \xstar + \radius \frac{x_0 - \xstar}{\|x_0 - \xstar\|_\glo} \in \partial \ball{\glo}{\xstar, \radius}$\footnote{Note that $\partial \ball{\glo}{\xstar, \radius}$ means the boundary of $\ball{\glo}{\xstar, \radius}$.}. The norm $\norm{y_0 - \xstar}_\glo = \radius$, therefore $y_0$ is a point on the boundary of $\ball{\glo}{\xstar, \radius}$. Moreover, the points $\xstar$, $y_0$, and $x_0$ lie on the same line segment, thus
\beq
\norm{x_0 - \xstar}_\glo = \norm{x_0 -y_0}_\glo + \radius.
\label{eq:aligned_points}
\eeq

By Lemma~\ref{lem:contraction_time}\ref{fact:contraction-time-mixed} and because each trajectory originating in $\ball{\glo}{\xstar, \radius}$ remains in $\mcS$, the $\rho$-contraction with respect to $\norm{\cdot}_\glo$ 
for the $\ce$-strongly contracting vector field $f$ is 
\begin{equation}
\label{eq:rho_time_proof}
t_\rho^{\loc,\glo}  = \frac{\ln(\ratiolg \rho^{-1})}{\ce}.
\end{equation}
Then, for almost every $t \in [0,t_\rho^{\loc,\glo}]$, we have
\begin{align}
    \|\odeflowtx{t}{x_0} - \xstar\|_\glo
    &\leq \|\odeflowtx{t}{x_0} - \phi_t(y_0)\|_\glo + \|\phi_t(y_0) - \xstar\|_\glo \label{eq:eq_proof_a} \\
    &\leq \|x_0 -y_0\|_\glo + \prodeqnorm_{\loc, \glo}\e^{-\ce t}\norm{y_0 - \xstar}_{\glo} \label{eq:eq_proof_b}\\
    &\overset{\eqref{eq:aligned_points}}{=} \|x_0 - \xstar\|_\glo - \|\xstar - y_0\|_\glo + \prodeqnorm_{\loc, \glo}\e^{-\ce t}\radius \nonumber\\
    &\!\!\! \overset{t = t_\rho^{\loc,\glo} }{\leq} \norm{x_0 - \xstar}_\glo - \radius(1 - \prodeqnorm_{\loc, \glo}\e^{-\ce{t_\rho^{\loc,\glo}}})\nonumber\\
    &\overset{\eqref{eq:rho_time_proof}}{=} \norm{x_0 - \xstar}_\glo - \radius(1 - \rho),\label{eq:decreasing1}
\end{align}
where in~\eqref{eq:eq_proof_a} we added and subtracted $\phi_t(y_0)$ and applied the triangle inequality, while inequality~\eqref{eq:eq_proof_b} follows from Assumption~\ref{ass:global-weak_contractivity} and inequality~\eqref{eq:condition_lsc_glo_diff_norms}.
Now,~\eqref{eq:decreasing1} implies that  $\norm{\odeflowtx{t_\rho^{\loc,\glo}}{x_0} - \xstar}_\glo \le \norm{x_0 - \xstar}_\glo - \radius(1 - \rho)$. If $\norm{x_0 - \xstar}_\glo - \radius(1 - \rho) \leq \radius$, then by Assumption~\ref{ass:local-strong_contractivity}, for almost every in $t\geq t_\rho^{\loc,\glo}$, we have
\begin{align*}
\norm{\odeflowtx{t}{x_0} - \xstar}_{\glo} & \leq \prodeqnorm_{\loc, \glo} \e^{-\ce(t - t_\rho^{\loc,\glo} )} \bigl(\norm{x_0 - \xstar}_\glo -  \radius(1 - \rho)\bigr).
\end{align*}
If $\norm{x_0 - \xstar}_\glo -  \radius(1 - \rho) > \radius$, we iterate the process. Specifically, let $x_\rho  := \phi_{t_\rho^{\loc,\glo} }(x_0)$, and define $\ds y_\rho := \xstar + \radius \frac{x_\rho - \xstar}{\|x_\rho  - \xstar\|_\glo} \in \partial \ball{\glo}{\xstar, \radius}$. Consider the solution to $\dot{y} = f(t, y)$ with initial condition $y({t_\rho^{\loc,\glo} }) = y_\rho$ and note that $\phi_t(x_\rho ) = \phi_{t+t_\rho^{\loc,\glo} }(x_0)$. For almost every $t \in [t_\rho^{\loc,\glo} , 2t_\rho^{\loc,\glo} ]$, we compute
\begin{align}
    \norm{\phi_{t+t_\rho^{\loc,\glo} }(x_0) - \xstar}_\glo &\leq \norm{\phi_t(x_\rho ) - \phi_t(y_\rho)}_\glo + \norm{\phi_t(y_\rho) - \xstar}_\glo \label{eq:eq_proof2_a} \\
    & \leq \|x_\rho  - y_\rho\|_\glo + \prodeqnorm_{\loc, \glo}\e^{-c (t - t_\rho^{\loc,\glo} )}\norm{y_0 - \xstar}_{\glo} \label{eq:eq_proof2_b}\\
    &\! \! \overset{\eqref{eq:aligned_points}}{=} \|x_\rho  - \xstar\|_\glo - \|\xstar - y_\rho\|_\glo + \prodeqnorm_{\loc, \glo}\e^{-c (t - t_\rho^{\loc,\glo} )}\radius \nonumber\\
    & \leq \norm{\phi_{t_\rho^{\loc,\glo} }(x_0) - \xstar}_\glo - \radius (1 - \prodeqnorm_{\loc, \glo}\e^{-c (t - t_\rho^{\loc,\glo} )})\nonumber\\
    &\! \! \overset{\eqref{eq:decreasing1}}{\leq} \norm{x_0 - \xstar}_\glo -  \radius(1 - \rho) - \radius (1 - \prodeqnorm_{\loc, \glo}\e^{-c (t - t_\rho^{\loc,\glo} )})\nonumber\\
    &\!\!\!\!\!\overset{t = 2t_\rho^{\loc,\glo} }{\leq}\norm{x_0 - \xstar}_\glo - 2 \radius(1 - \rho),\nonumber
\end{align}
where in~\eqref{eq:eq_proof2_a} we added and subtracted $\phi_t(y_0)$ and applied the triangle inequality, while~\eqref{eq:eq_proof_b} follows from Assumption~\ref{ass:global-weak_contractivity} and~\eqref{eq:condition_lsc_glo_diff_norms}.
We now reason as done in $[0, t_\rho^{\loc,\glo} ]$. If $\norm{x_0 - \xstar}_\glo - 2 \radius(1 - \rho) \leq \radius$, then Assumption~\ref{ass:local-strong_contractivity} implies
$$\norm{\phi_{t+t_\rho^{\loc,\glo} }(x_0) - \xstar}_{\glo} \leq \prodeqnorm_{\loc, \glo}\bigl(\norm{x_0 - \xstar}_\glo - 2 \radius(1 - \rho)\bigr)\e^{-c(t - 2t_\rho^{\loc,\glo} )}, \quad \forall t\geq 2t_\rho^{\loc,\glo} .$$
If $\norm{x_0 {-} \xstar}_\glo {-} 2 \radius(1 {-} \rho) > \radius$, we proceed analogously until $\norm{x_0 - \xstar}_\glo - T \radius(1 - \rho) \leq \radius$. This inequality is verified after at most $T := \Bigceil{\frac{\glbnorm{x_0-\xstar} - \radius}{(1-\rho)\radius}}$ steps. Iterating the previous process, at step $T$, for almost every $t \in [(T-1)t_\rho^{\loc,\glo} , Tt_\rho^{\loc,\glo} ]$, we get
\begin{align}
    \norm{\phi_{t+(T-1)t_\rho^{\loc,\glo} }(x_0) - \xstar}_\glo &\leq \norm{x_0 - \xstar}_\glo - (T-1) \radius(1 - \rho) - \radius (1 - \prodeqnorm_{\loc, \glo}\e^{-c (t - (T-1)t_\rho^{\loc,\glo} )}),\nonumber\\
    &\!\!\!\!\overset{t = kt_\rho^{\loc,\glo}}{\leq}\!\!\!\!\norm{x_0 - \xstar}_\glo - T \radius(1 - \rho) \leq \radius,\nonumber
\end{align}
where the last inequality follows from the definition of $T$.
Local strong contractivity then implies
$$\norm{\phi_{t+Tt_\rho^{\loc,\glo} }(x_0) - \xstar}_{\glo} \leq \prodeqnorm_{\loc, \glo}\bigl(\norm{x_0 - \xstar}_\glo - T \radius(1 - \rho)\bigr)\e^{-c(t - Tt_\rho^{\loc,\glo} )}, \quad \text{ for almost every } t\geq Tt_\rho^{\loc,\glo} .$$
The above reasoning together with Assumption~\ref{ass:global-weak_contractivity} implies that for almost every $t \in [i t_\rho^{\loc,\glo}, (i+1) t_\rho^{\loc,\glo} ]$, $i \in \{0,\dots, T-1\}$, we have
\begin{align}
\label{proof:ineq_key_passage}
\norm{\phi_{t + it_\rho^{\loc,\glo} }(x_0) - \xstar}_\glo \leq\min\Big\{\norm{x_0 - \xstar}_\glo - i  \radius(1 - \rho), \norm{x_0 - \xstar}_\glo - i  \radius(1 - \rho) - \radius\bigl(1-\prodeqnorm_{\loc, \glo}\e^{-c(t - i {t_\rho^{\loc,\glo} })}\bigr)\Big\}.
\end{align}
By partitioning the time interval $[0, +\infty[$ as $[0,t_\rho^{\loc,\glo} [ \cup \dots \cup [(T-1)t_\rho^{\loc,\glo} , T t_\rho^{\loc,\glo} [ \cup [T t_\rho^{\loc,\glo} , +\infty[$ and summing up the above inequalities we obtain the bound:
\begin{align}
\norm{\phi_{t}(x_0) - \xstar}_{\glo}& \leq \sum_{i = 0}^{T-1} \indicator{\{i{t_\rho^{\loc,\glo} } \leq t < (i+1){t_\rho^{\loc,\glo} }\}}(t) \cdot \min\Big\{\norm{x_0 - \xstar}_\glo - i  \radius(1 - \rho), \norm{x_0 - \xstar}_\glo - i  \radius(1 - \rho)\nonumber\\
& \quad- \radius\bigl(1-\prodeqnorm_{\loc, \glo}\e^{-c(t - i {t_\rho^{\loc,\glo} })}\bigr)\Big\} + \indicator{\{t \geq T {t_\rho^{\loc,\glo} }\}} (t) \cdot\min\Big\{\norm{x_0 - \xstar}_\glo - T \radius(1 - \rho), \nonumber\\
&\quad \prodeqnorm_{\loc, \glo} \left(\norm{x_0 - \xstar}_\glo - T \radius(1 - \rho)\right)\e^{-c(t - T {t_\rho^{\loc,\glo} })}\Big\} := \subscr{g}{B}(t) \label{eq:g_b}.
\end{align}
Finally, item~\ref{item:2_diff_norm} follows by noticing that $\subscr{g}{B}(t) \leq \linexp{t \; ;\; q, \cl, \ce, \tld}$, $t \geq 0$, for the values $\tld = T{\ln(\ratiolg \rho^{-1})}/{\ce} + \ln(\ratiolg) / \ce$, $\cl = \radius\,\ce(1-\rho)/\ln(\prodeqnorm_{\loc,\glo}\rho^{-1})$, and $\ds q = \norm{x_0 - \xstar}_\glo + \radius(1 - \rho)\frac{\ln(\ratiolg)}{\ln(\ratiolg \rho^{-1})}$.
This concludes the proof.
\end{proof}

\medskip
\begin{rem}
\label{rem:comparison}
Theorem~\ref{thm:lin_exp_decay_GWC_and_LSC_dynamics_diff_norms} sharpens the convergence bound  in~\cite[{Theorem 4 and Corollary 2}]{VC-AG-AD-GR-FB:23a}.
{This improvement stems from a more accurate intercept and linear-exponential crossing time. Differently from~\cite{VC-AG-AD-GR-FB:23a}, the bound in Theorem~\ref{thm:lin_exp_decay_GWC_and_LSC_dynamics_diff_norms} remains continuous at all times, and no jump can occur at $\tld$. Regarding the proof techniques compared to~\cite{VC-AG-AD-GR-FB:23a}, please note that the pivotal point in the proof leading to the sharper bound is inequality~\eqref{proof:ineq_key_passage}.}
An illustration between the bound in~\eqref{eq:bound_outside_the_ball_diff_norms} and the one in~\cite{VC-AG-AD-GR-FB:23a} is given in Figure~\ref{fig:comparison_bounds}.
\end{rem}

\begin{SCfigure}[5][!ht]
\centering
\includegraphics[width=.65\linewidth]{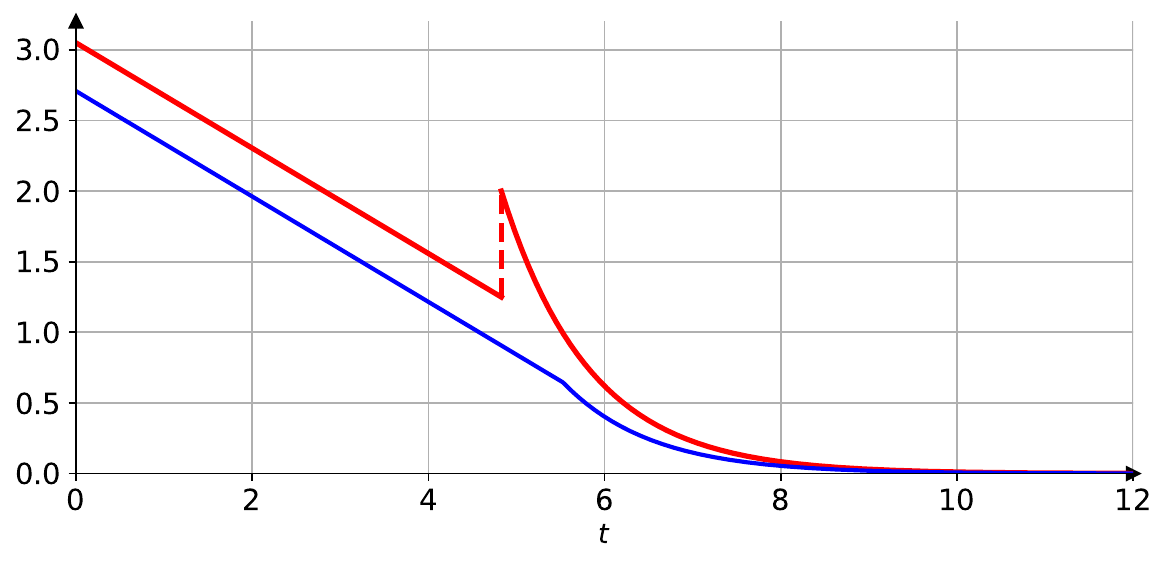}
\caption{Linear-exponential bound in equation~\eqref{eq:bound_outside_the_ball_diff_norms} (solid blue curve)
and the decay bound presented in~\cite{VC-AG-AD-GR-FB:23a} (red curve) for $\norm{x_0 - \xstar}_\glo = 2.4$, $\radius = 1$, $\ce = 1$, $\prodeqnorm_{L,G} = 2$, $\rho = 0.4$.}
\label{fig:comparison_bounds}
\end{SCfigure}

\begin{rem}
A consequence of Theorem~\ref{thm:lin_exp_decay_GWC_and_LSC_dynamics_diff_norms} is that, jointly, Assumptions (A1), (A2), and (A3) preclude the existence of any other invariant sets besides $\mcS$, and the convergence towards the equilibrium is global.
\end{rem}
\begin{rem}
Linear-exponential convergence is weaker than global exponential convergence, but stronger than global asymptotic convergence (as, e.g., we provide an explicit estimate of the time required to reach a neighborhood of the equilibrium).
\end{rem}

\begin{rem}
The bound in Theorem~\ref{thm:lin_exp_decay_GWC_and_LSC_dynamics_diff_norms} generalizes the result for equal norms in Theorem~\ref{thm:local-exp-contractivity_same_norms}. In fact, the factor $(1-\rho)/\ln(\prodeqnorm_{\loc,\glo} \rho^{-1})$ is always less than 1 for $\ratiolg>1$. Moreover, when $\ratiolg = 1$ it results $\ds \lim_{\rho \to 1} (1-\rho)/\ln(\prodeqnorm_{\loc,\glo} \rho^{-1}) = 1$, thereby exactly recovering the equal-norm result.
\end{rem}

\section{Local Stability in the Presence of External Inputs}
\label{2024a-sec:external_input}
We now characterize local ISS for \gwlsc systems w.r.t. the same norm. Specifically, we consider the system
\beq
\label{eq:dynamical_system_input}
\dot x(t) = f\bigl(t, x(t), u(t)\bigr).
\eeq
where, $\map{f}{\R_{\geq0} \times \mcC \times \mcU }{\R^n}$, the map $x \mapsto f(t, x, u)$ is locally Lipschitz, for all $t$, $u$, with $\mcC \subseteq \R^n$ $f$-invariant, open and convex, and $\mcU \subset \R^m$. 
Given $\bar u \in \R^m$, we define the set of bounded inputs $\bar \mcU := \setdef{\map{u}{\R_{\geq0}}{\mcU}}{\norm{u(t)}_\mcU \leq \bar u, \forall t \geq 0}$. We make the following:
\begin{assumptionbox}
\label{ass:set_assumption_input}
there exist norms $\|\cdot\|, \|\cdot\|_{\mcU}$ on $\mcC$ and $\mcU$, respectively, such that
\begin{enumerate}[label=\textup{($A$\arabic*')}, leftmargin=1.2 cm,noitemsep]
\item \label{ass:global-weak_contractivity_input}
for all $t$, $u$, the map $x \mapsto f(t, x, u)$ is weakly infinitesimally contracting on $\R^n$ w.r.t. $\norm{\cdot}$;
\item \label{ass:lipschitz_input}
for all $t$, $x$, the map $u \mapsto f(t, x, u)$ is Lipschitz with constant $L_u \geq 0$;
\item \label{ass:local-strong_contractivity_input}
there exist a forward-invariant set $ \mcS$ and $\ce >0$ such that, 
for all $t$, for each $u \in \bar \mcU$, the map $x \mapsto f(t, x, u(t))$ is $\ce$-strongly infinitesimally contracting on $\mcS$ w.r.t. $\norm{\cdot}$;
\item \label{ass:eq_point_input}
at $u(t) = \0_m$, for all $t$, there exists an equilibrium point $\xstar \in  \mcS$.
\end{enumerate}
\end{assumptionbox}
We begin by giving two technical lemmas, needed to prove the main result of this section.
\smallskip
\begin{lem}[Error dynamics for input-dependent systems]
\label{lem:general_responce}
Consider system~\eqref{eq:dynamical_system_input} satisfying Assumption~\ref{ass:lipschitz_input}. Then any two solutions $x(t)$ and $y(t)$ with input $\map{u_x, u_y}{\R_{\geq 0}}{\R^m}$, satisfy for almost every $t\geq 0$,
\begin{align}
D^+ \norm{x(t) - y(t)} & \leq \int_0^1\mu\Bigl( Df(y + \alpha(x - y), u_y)\Bigr)d\alpha  \norm{x(t) - y(t)}  + L_u \norm{u_x(t) - u_y(t)}_{\mcU}.\label{eq:bound_mean_value+input}
\end{align}
\end{lem}
\begin{proof}
Let $x(t)$ and $y(t)$ be two trajectories of~\eqref{eq:dynamical_system_input} with input signals $u_x, u_y$, respectively. Let $\WP{\cdot}{\cdot}$ be a weak pairing compatible with $\norm{\cdot}$. 
We compute
\begin{align}
\norm{x(t) - y(t)} &D^+ \norm{x(t)-y(t)} = \WP{f(t, x, u_x) - f(t,y, u_y)}{x - y} \label{proof_1_}\\
&\leq \WP{f(t, x, u_y) - f(t,y, u_y)}{x - y} + \norm{f(t, x, u_x) - f(t, x, u_y)}\norm{x - y} \label{proof_2_}\\
&\leq \WP{f(t, x, u_y) - f(t,y, u_y)}{x - y} + L_u \norm{u_x - u_y}_{\mcU}\norm{x - y} \label{proof_3_},
\end{align}
where in~\eqref{proof_1_} we used the curve norm derivative formula~\ref{prop:wp:curve_derivative}, in~\eqref{proof_2_} we added and subtracted $f(t,x, u_y)$ and used the sub-additivity~\ref{prop:wp:sub-additivity} and the Cauchy-Schwartz inequality~\ref{prop:wp:cauchy-schwartz}, and in~\eqref{proof_3_} we used Assumption~\ref{ass:lipschitz_input}.

Next, by dividing both side for $\norm{x(t) - y(t)}$ we get
\begin{align}
D^+ \norm{x(t) - y(t)} &\leq \frac{\WP{f(t, x, u_y) - f(t,y, u_y)}{x - y}}{\norm{x - y}} + L_u \norm{u_x - u_y}_{\mcU}\nonumber\\
&= \frac{\WP{f(t, x, u_y) - f(t,y, u_y)}{x - y}}{\norm{x - y}^2}\norm{x - y} + L_u \norm{u_x - u_y}_{\mcU}.\label{proof_wp_4}
\end{align}

By applying the mean-value Theorem~\ref{thm:mean_value_thm} to~\eqref{proof_wp_4}, a.e., we get
\begin{align}
D^+ \norm{x(t) - y(t)} & \leq \frac{\WP{\int_0^1 Df(y + s(x - y), u_y)ds( x - y)}{ x - y}}{\norm{ x - y} } \frac{\norm{x - y}}{\norm{ x - y}} + L_u \norm{u_x - u_y}_{\mcU}\label{proof_wp_5}\\
& \leq \int_0^1\frac{\WP{ Df(y + s( x - y), u_y)ds(x- y)}{x - y}}{\norm{x - y}^2}ds \norm{x - y} + L_u \norm{u_x - u_y}_{\mcU},\label{proof_wp_6}
\end{align}
where in~\eqref{proof_wp_5} we have used the weak pairing sub-additivity~\ref{prop:wp:sub-additivity}.
Next, recall that Lumer’s equality~\ref{prop:wp:lumers_eq} implies, $ \ds \frac{\WP{Az}{z}}{\WP{z}{z}} \leq \mu\bigl(A\bigr)$ for every $A\in \R^{n\times n}$ and $z \neq \0_n$. By applying this equality to~\eqref{proof_wp_6} (with $A = Df(y + s(x - y), u_y)$ and $z = x - y$) we get inequality~\eqref{eq:bound_mean_value+input}. This concludes the proof.
\end{proof}
\smallskip
The next result gives a linear-exponential bound for the solution of dynamics with saturations and additive inputs.
\smallskip
\begin{lem}[Solution of dynamics with saturations and additive inputs]
\label{lem:linear-decay_u}
Let $\ce$ and $d$ be positive scalars, and $\map{u}{\R_{\geq 0}}{\R^n}$ satisfying $\ds \norm{u(t)}_{\infty} = \subscr{u}{max} < d \ce$, for all $t$. Consider the dynamics
\beq
\label{eq:sat_input}
\dot x(t) =  - \ce \sat{d}{x(t)} + u(t), \quad x_0 = q >d.
\eeq
Then, a solution of~\eqref{eq:sat_input} satisfies
$$
x(t) \leq \linexp{t \; ;\; q, \cl, \ce, \tld} + \indicator{[\tld, +\infty[}(t)(1 - \e^{- \ce  (t - \tld)})\frac{\subscr{u}{max}}{\ce},
$$
with $\cl := d \ce - \subscr{u}{max} > 0$ and $\tld := \frac{q - d}{\cl} > 0$.
\end{lem}
\smallskip
\begin{proof}
Using the definition of saturation function, for all $t \in \R_{\geq 0}$ we can upper bound the ODE~\eqref{eq:sat_input} as
\beq
\label{eq:LED:ode_proof_u}
\dot x(t)  \leq \dot y(t) := 
\begin{cases}
- d \ce + \subscr{u}{max} & \textup{ if } y(t) > d,\\
- \ce x(t) + \subscr{u}{max} & \textup{ if } y(t) \in [-d, d],\\
d \ce + \subscr{u}{max} & \textup{ if } y(t) < -d,\\
\end{cases}
\eeq
which, in each interval $[t_0, t_1] \subseteq \R_{\geq 0}$ where the solution is continuous and does not change regime, has general solution
\begin{align}
\label{eq:LED:ode_proof_sol_u}
y(t)
& =
\begin{cases}
\ds (- d \ce  + \subscr{u}{max})t + y(t_0) & \textup{ if } y(t) > d,\\
\ds \Bigl(y(t_0) - \frac{\bar u}{\ce}\Bigr)\e^{- \ce  (t - t_0)} + \frac{\subscr{u}{max}}{\ce} & \textup{ if } y(t) \in [-d, d],\\
\ds (d \ce  + \subscr{u}{max})t + y(t_0) & \textup{ if } y(t) < - d.\\
\end{cases}
\end{align}
At time $t = 0$, we have $x_0 = q >d$. For continuity of the solution, there exists $\tstar$ such that $y(t) > d$ for all $t \in [0,\tstar]$. Thus from~\eqref{eq:LED:ode_proof_sol_u} and being $x(t_0 = 0) = q$, it is 
$$
y(t) = (- d \ce + \subscr{u}{max})t + q,
$$
for all $t \in [0,\tstar]$. Moreover being $\subscr{u}{max} < d \ce$, the function $y(t)$ is decreasing, the time value $\tstar$ is finite and there exists a time, say it $\bar t$ such that $y(\bar t) = d$. Let $\cl : = d \ce - \subscr{u}{max}$, we have
\[
y(\bar t) = d \iff  - \cl \bar t + q = d  \iff \bar t = \frac{q - d}{\cl} := \tld.
\]

In summary, we have shown that $y(t) = q - \cl t $ for all $t \in [0,\tld]$ and is equal to $d$ at time $\tld$. Therefore from~\eqref{eq:LED:ode_proof_sol_u} and being $y(\tld) = q - \cl \tld$, for all time $t > \tld$ we have 
$$
y(t) = \bigl(q - \cl \tld\bigr) \e^{- \ce  (t - \tld)}+ \bigl(1 - \e^{- \ce  (t - \tld)}\bigr)\frac{\subscr{u}{max}}{\ce}.
$$
Specifically, $y(t) >0$ for all $t > \tld$, thus it can never be $y(t) < -d$.
This concludes the proof.
\end{proof}
\smallskip
We are now ready to state the following:
\smallskip
\bt[Local ISS for input-dependent \gwlsc systems]\label{lem:sat_error_input}
Consider system~\eqref{eq:dynamical_system_input} satisfying Assumptions~\ref{ass:global-weak_contractivity_input} --~\ref{ass:eq_point_input}. Let $\radius$ be the largest radius such that $\ball{}{\xstar, \radius} \subseteq \mcS$, $\bar u < \radius \ce$, and $\subscr{u}{max} := \sup_{\tau \in [0,t]}\norm{u_x(\tau)}_\mcU \leq \bar{u}$.
For each trajectory $x(t)$ with input $u_x \in \bar \mcU$ starting from $x_0 \notin \mcS$, for almost every $t\geq 0$, we have:
\begin{enumerate}
\smallskip
\item \label{eq:bound_dini_input}
$
\ds D^+ \norm{x(t) {-} \xstar} \leq -\ce \satur_{\radius} (\norm{x(t){-}\xstar}) {+} L_u \norm{u_x(t)}_{\mcU};
$
\smallskip
\item
\label{eq:lin-exp_input}
$
\ds \norm{x(t) {-} \xstar} \leq \linexp{t \; ;\; q, \cl, \ce, \tld} + \indicator{[\tld, +\infty[}(t)\frac{L_u}{\ce} (1 - \e^{-\ce t})\subscr{u}{max},
$

with
\bi
\item exponential decay rate $\ce > 0$;
\item intercept $\ds q = \norm{x_0-\xstar}$;
\item linear decay rate $\cl =  \radius \ce - \subscr{u}{max}$;
\item linear-exponential crossing time $\tld ={(q - \radius)}/{\cl}$.
\ei
\end{enumerate}
\et
\begin{proof}
Consider an arbitrary trajectory $x(t)$ starting from $x_0 \notin \mcS$ with input $u_x$ and a second trajectory equal to the equilibrium $\xstar$ with input $u = \0_m$. To prove statement~\ref{eq:bound_dini_input}, let $\mu$ be the log-norm associated to $\norm{\cdot}$. By applying inequality~\eqref{eq:bound_mean_value+input} to those trajectories, for almost every $t \geq 0$, we have
\begin{align}
D^+ \norm{x(t) - \xstar} & \leq \int_0^1\mu\Bigl( Df(\xstar + \alpha(x(t) - \xstar), 0)\Bigr)d\alpha \norm{x(t) - \xstar} + L_u \norm{u_x}_{\mcU}.
\end{align}
The proof follows by using similar reasoning as the one in the proof of Lemma~\ref{lem:saturated_error_dynamics}.
Item~\ref{eq:lin-exp_input} follows by using the Comparison Lemma~\cite[pp.~102-103]{HKK:02} and Lemma~\ref{lem:linear-decay_u} to upper bound the solution to the differential inequality~\ref{eq:bound_dini_input}.
\end{proof}

\section{Tackling Linear Programs}
\label{2024a-sec:application}
We now show the efficacy of the previous results by applying them to a dynamical system solving the LP problem.
Given $c\in \R^n$, $A \in \R^{m\times n}$ and $b\in \R^m$, we consider the \emph{linear program}:
\beq\label{eq:lin_program}
\begin{aligned}
\min_{x \in \R^n} \quad & c^\top x, \\
\text{s.t.} \quad & Ax \leq b,
\end{aligned}
\eeq
and its equivalent unconstrained formulation
\beq
\label{eq:lin_program_unconstrained}
\min_{x \in \R^n} c^\top x + \iota_{\mathcal{I}_b}(Ax),
\eeq
where $\mathcal{I}_b = \setdef{y \in \R^m}{y - b \leq \0_m}$. We assume that~\eqref{eq:lin_program_unconstrained} admits a unique equilibrium.
Note that~\eqref{eq:lin_program_unconstrained} is a particular composite minimization problem:
\begin{equation}\label{eq:composite_minimization}
\min_{x \in \R^n} f(x) + g(Ax),
\end{equation}
with $f(x) = c^\top x$ and $g(Ax) = \iota_{\mathcal{I}_b}(Ax)$.
To solve~\eqref{eq:lin_program_unconstrained}, we leverage the proximal augmented Lagrangian approach proposed in~\cite{NKD-SZK-MRJ:19} and consider the \emph{proximal augmented Lagrangian}, $\map{\tilde L_\gamma}{\R^n \times \R^m}{\R}$, defined by
\begin{equation}
\label{eq:proximal_augmented_Lagrangian}
\tilde L_\gamma(x,\lambda) = f(x) + M_{\gamma g}(Ax + \gamma \lambda) - \frac{\gamma}{2}\|\lambda\|_2^2,
\end{equation}
where $\lambda \in \R^m$ is the Lagrange multiplier, $\gamma > 0$ is a parameter, and $M_{\gamma g}$ is Moreau envelope of $g$.
\smallskip
\begin{rem}
For $f$ continuously differentiable, convex, and with a Lipschitz continuous gradient, and $g$ convex, closed and proper, solving the composite minimization problem~\eqref{eq:composite_minimization} corresponds to finding saddle points of~\eqref{eq:proximal_augmented_Lagrangian}, simultaneously updating the primal and dual variables~\cite[Theorem~2]{NKD-SZK-MRJ:19}.
\end{rem}
\smallskip
Next, consider the \emph{continuous-time augmented primal-dual dynamics}~\cite{NKD-SZK-MRJ:19} (that can be interpreted as a continuous-time neural network) associated to the proximal augmented Lagrangian of problem~\eqref{eq:lin_program_unconstrained}
\beq
\label{eq:prox-aug-primal-dual_lin_progr}
\begin{aligned}
\dot{x} &= -\nabla_x \tilde L_\gamma(x,\lambda) = - c - A^\top \nabla M_{\gamma \iota_{\mathcal{I}_b}}(Ax + \gamma \lambda) = - c - \frac{1}{\gamma}A^\top \relu\bigl( Ax + \gamma \lambda - b\bigr),\\
\dot{\lambda} &= \nabla_{\lambda} \tilde L_\gamma(x,\lambda) =  - \gamma\lambda + \gamma\nabla M_{\gamma \iota_{\mathcal{I}_b}}(Ax + \gamma \lambda) = - \gamma \lambda + \relu\bigl( Ax + \gamma \lambda - b\bigr).
\end{aligned}
\eeq
We let $\map{\subscr{F}{LP}}{\R^{n+m}}{\R^{n+m}}$ denote the vector field for~\eqref{eq:prox-aug-primal-dual_lin_progr}.
\smallskip
\begin{rem}
Equation~\eqref{eq:prox-aug-primal-dual_lin_progr} follows directly after noticing that for almost every $y\in \R^m$ it results 
\begin{align*}
\nabla M_{\gamma \iota_{\mathcal{I}}}(y) & = \frac{1}{\gamma}\bigl(y - \proj_{\iota_{\mathcal{I}}}(y) \bigr)= \frac{1}{\gamma}\bigl(y - \min \{y , b \} \bigr) =  \frac{1}{\gamma}\relu\bigl(y - b\bigr).
\end{align*}
\end{rem}
The next result characterizes the convergence of~\eqref{eq:prox-aug-primal-dual_lin_progr}.
\smallskip
\bt[Convergence of the linear program]
\label{thm:contr_lin_prog}
Consider the dynamics~\eqref{eq:prox-aug-primal-dual_lin_progr} and let $(\xstar, \lambda^\star)\in \R^{n+m}$ be an equilibrium point. If $D\subscr{F}{LP}\bigl(\xstar,\lambda^\star)$ is Hurwitz, then any solution of~\eqref{eq:prox-aug-primal-dual_lin_progr} linear-exponentially converges towards $(\xstar, \lambda^\star)$.
\et
\begin{proof}
To prove the statement we show that~\eqref{eq:prox-aug-primal-dual_lin_progr} satisfies the assumptions of Theorem~\ref{thm:lin_exp_decay_GWC_and_LSC_dynamics_diff_norms}. First, we prove that the system is globally-weakly contracting. To this purpose, let $z := (x,\lambda) \in \R^{n+m}$, $y := Ax + \gamma \lambda - b$ and define $G(y) := D\relu\bigl(y\bigr)$, for almost every $y \in \R^m$. The Jacobian of~\eqref{eq:prox-aug-primal-dual_lin_progr} is
\[
D\subscr{F}{LP}(z) =
\begin{bmatrix}
- \frac{1}{\gamma}A^\top G(y)A & \!\!- A^\top G(y) \\
G(y)A & \!\! -\gamma (I_m - G(y))
\end{bmatrix}.
\]
Being $0 \preceq G(y) \preceq I_m$\footnote{For every $\gamma >0$, $0 \preceq \nabla^2 M_{\gamma g}(y) \preceq \frac{1}{\gamma} I_n$, a.e. $y \in \R^m$~\cite[Lemma 18]{AD-VC-AG-GR-FB:23f}.}, a.e. $y \in \R^m$, we have
\begin{align*}
\sup_{z} \mu_2 \bigl(D\subscr{F}{LP}(z)\bigr) \leq \max_{0 \preceq G \preceq I_m} \mu_2 \left(
\begin{bmatrix}
- \gamma^{-1} A^\top GA & - A^\top G \\
GA & \gamma \bigl(G - I_m\bigr)
\end{bmatrix}\right),
\end{align*}
By definition of $\mu_2$, we have that
\begin{align*}
\mu_2 \left(
\begin{bmatrix}
- \gamma^{-1} A^\top GA & - A^\top G \\
GA & \gamma \bigl(G - I_m\bigr)
\end{bmatrix}\right)  &=
\subscr{\lambda}{max}\Biggl(
\begin{bmatrix}
- \gamma^{-1} A^\top GA & 0 \\
0 & \gamma \bigl(G - I_m\bigr)
\end{bmatrix} 
\Biggr)\\
&= \max\{\subscr{\lambda}{max}\bigl(- \gamma^{-1}A^\top GA\bigr), \subscr{\lambda}{max}\bigl(\gamma \bigl(G - I_m\bigr)\bigr)\} \leq 0.
\end{align*}
The last equality follows from the fact that $\subscr{\lambda}{max}\bigl(- \gamma^{-1}A^\top GA\bigr) = \subscr{\lambda}{max}\bigl(\gamma \bigl(G - I_m\bigr)\bigr) \leq 0$. In particular, the equality $\subscr{\lambda}{max}\bigl(-\gamma \bigl(G - I_m\bigr)\bigr) \leq 0$ follows directly from $0 \preceq G \preceq I_m$; while $\subscr{\lambda}{max}\bigl(- \gamma^{-1} A^\top GA\bigr) \leq 0$, follows noticing that $A^\top GA \succeq 0$\footnote{$A^\top GA \succeq 0 \iff x^\top A^\top G A x \geq 0 \quad \forall x \in \R^n \iff y^\top G y \geq 0$, $\forall y \in \R^m \iff G \succeq 0$.}. This implies that~\eqref{eq:prox-aug-primal-dual_lin_progr} is weakly contracting on $\R^{n+m}$ w.r.t. $\norm{\cdot}_2$.
Thus~\eqref{eq:prox-aug-primal-dual_lin_progr} is weakly contracting on $\R^{n+m}$ {with respect to} $\norm{\cdot}_2$. Next, we prove that the system is locally-strongly contracting. 
To do so, we first note that for any equilibrium point $z^\star := (\xstar, \lambda^\star)$ of~\eqref{eq:prox-aug-primal-dual_lin_progr}, {both $D\relu(y^\star)$ and $D\subscr{F}{LP}(z^\star)$ are differentiable in a neighborhood of $y^\star$ and $z^\star$, respectively}.
{In fact, for each $i$, the KKT conditions ensures that either $(A\xstar)_i - b_i = 0$ or $\lambda_i^{\star} = 0$.  In turn, this implies that $y^{\star}_i = (A\xstar)_i + \gamma {\lambda}^{\star}_i - b_i \neq 0$, for all $i$.}
Now, being by assumption $D\subscr{F}{LP}(z^\star)$ Hurwitz, there exists $Q$ invertible such that $\wlognorm{2}{Q}{D\subscr{F}{LP}(z^\star)}<0$~\cite[Corollary 2.33]{FB:23-CTDS}.
Let $\mathcal{K}$ be the set of differentiable points in a neighborhood of $z^\star$. Then, by the continuity property of the log-norm, there exists $\ball{2,Q}{z^\star, p}$, with $p := \sup \setdef{p > 0}{\ball{2,Q}{z^\star, p} \subset \mathcal{K}}$, where $D\subscr{F}{LP}(z)$ exists and $\wlognorm{2}{Q}{D\subscr{F}{LP}(z)}< - \ce$ for all $z \in \ball{2,Q}{z^\star, p}$, for some $\ce >0$.
Therefore~\eqref{eq:prox-aug-primal-dual_lin_progr} is strongly infinitesimally contracting w.r.t. $\norm{\cdot}_{2,Q}$ in $\ball{2,Q}{z^\star, p}$. This concludes the proof.
\end{proof}
\smallskip
A key hypothesis of Theorem~\ref{thm:contr_lin_prog} is that ${D\!\subscr{F}{LP}({\xstar\!,\lambda^\star})}$ is Hurwitz. {This hypothesis can only be verified by prior knowledge of the LP solution. This limitation motivates the following conjecture, which would relate stability of ${D\!\subscr{F}{LP}({\xstar\!,\lambda^\star})}$ to matrix $A$ and the KKT conditions.} 
\smallskip
\begin{conj}
\label{conj_lin_prog}
Let $(\xstar, \lambda^\star)$ be the equilibrium of~\eqref{eq:prox-aug-primal-dual_lin_progr}. The LP~\eqref{eq:lin_program} has a unique solution, $\xstar$, if and only if $D\subscr{F}{LP}\bigl(\xstar,\lambda^\star)$ is Hurwitz.
\end{conj}
\paragraph{Numerical Experiments}
Consider the following LP
\beq\label{eq:lin_program_eq}
\begin{aligned}
\min_{x \in \R^3} \quad & x_1 + x_2 + x_3, \\
\text{s.t.} \quad & -1 \leq x_1 \leq 1, -1 \leq x_2 \leq 1, -1 \leq x_3 \leq 1.
\end{aligned}
\eeq
for which the unique optimal solution is $\xstar = (-1, -1, -1)$.

Next, consider the corresponding continuous-time augmented primal-dual dynamics~\eqref{eq:prox-aug-primal-dual_lin_progr}
\beq
\begin{aligned}
\label{eq:prox-aug-primal-dual_lin_progr_ex}
\dot{x}_1 &= -1 - \frac{1}{\gamma}\Bigl(\relu\left(x_1 + \gamma \lambda_1 - 1\right) - \relu\left(-x_1 + \gamma \lambda_4 - 1\right)\Bigr), \\
\dot{x}_2 &= -1 - \frac{1}{\gamma}\Bigl(\relu\left(x_2 + \gamma \lambda_2 - 1\right) - \relu\left(-x_2 + \gamma \lambda_5 - 1\right)\Bigr), \\
\dot{x}_3 &= -1 - \frac{1}{\gamma}\Bigl(\relu\left(x_3 + \gamma \lambda_3 - 1\right) - \relu\left(-x_3 + \gamma \lambda_6 - 1\right)\Bigr), \\
\dot{\lambda}_1 &= -\gamma \lambda_1 + \relu\left(x_1 + \gamma \lambda_1 - 1\right) \\
\dot{\lambda}_2 &= -\gamma \lambda_2 + \relu\left(x_2 + \gamma \lambda_2 - 1\right), \\
\dot{\lambda}_3 &= -\gamma \lambda_3 + \relu\left(x_3 + \gamma \lambda_3 - 1\right), \\
\dot{\lambda}_4 &= -\gamma \lambda_4 + \relu\left(-x_1 + \gamma \lambda_4 - 1\right), \\
\dot{\lambda}_5 &= -\gamma \lambda_5 + \relu\left(-x_2 + \gamma \lambda_5 - 1\right), \\
\dot{\lambda}_6 &= -\gamma \lambda_6 + \relu\left(-x_3 + \gamma \lambda_6 - 1\right).
\end{aligned}
\eeq
We set $\gamma = 0.5$ and simulate the dynamics~\eqref{eq:prox-aug-primal-dual_lin_progr} over the time interval $t \in [0, 40]$ with a forward Euler discretization with step-size $\Delta t = 0.001$, starting from $150$ initial conditions generated as follows: we first randomly generate an initial condition and then define the remaining 149 initial conditions by adding, to the first initial condition, random noise generated from a normal distribution with mean $0$ and standard deviation $2$. The simulation results are that each resulting trajectory converges to the point $z^\star = (-1, -1, -1, 0, 0, 0, 1, 1, 1)$. Next, we numerically found that $DF_{LP}(z^\star)$ is Hurwitz (in alignment with our conjecture).
Figure~\ref{fig:lin_prog_log} illustrates the mean and standard deviation of the lognorm of the $\ell_2$ distance of the 150 simulated trajectories of~\eqref{eq:prox-aug-primal-dual_lin_progr} w.r.t. $z^\star$. In agreement with Theorem~\ref{thm:contr_lin_prog} the convergence is linearly-exponentially bounded.
\begin{figure}[!h]
\centering
\includegraphics[width=0.7\linewidth]{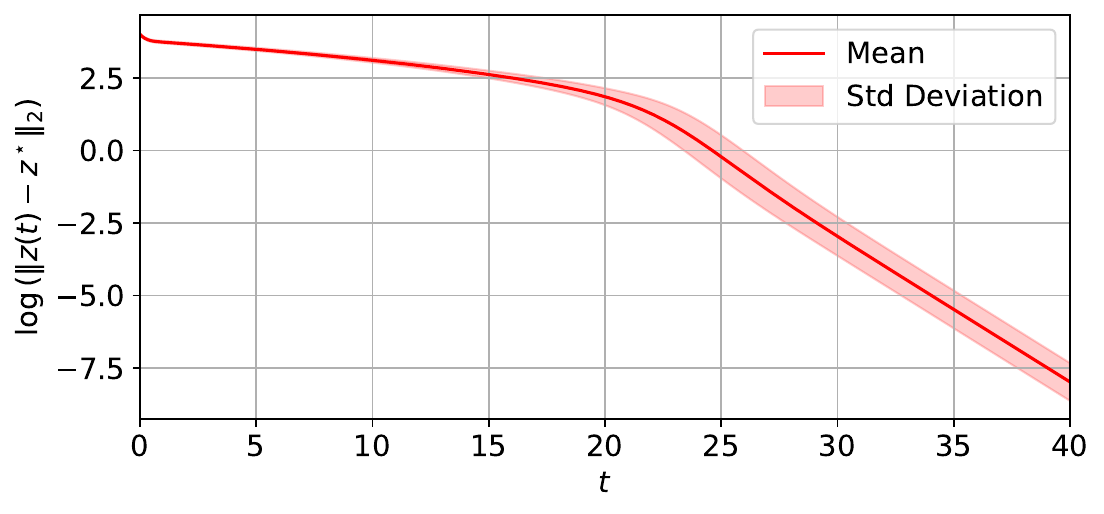}
\caption{Mean (red curve) and standard deviation (shadow curve) of the lognorm of the Euclidean distance of 150 simulated trajectories of~\eqref{eq:prox-aug-primal-dual_lin_progr} with respect to the equilibrium point $z^\star$. In agreement with Theorem~\ref{thm:contr_lin_prog} the convergence is linearly-exponentially bounded.}
\label{fig:lin_prog_log}
\end{figure}

\section{Conclusion}
\label{2024a-sec:conclusion}
We analyzed the convergence characteristics of \gwlsc dynamics, which naturally arise from convex optimization problems with a unique minimizer. For such dynamics, we
showed linear-exponential convergence to the equilibrium.
Specifically, we demonstrated that linear-exponential dependency arises naturally in certain dynamics with saturations and used this result for our convergence analysis. Depending on the norms where the system is \gwlsc, we considered two different scenarios that required two distinct mathematical approaches, yielding convergence bounds that are sharper than those in~\cite{VC-AG-AD-GR-FB:23a}. Finally, after giving a sufficient condition for local ISS, we illustrated our results on the continuous-time augmented primal-dual dynamics solving LPs. Our results motivated a conjecture relating the optimal solution of LPs to the local stability properties of the equilibrium of the resulting dynamics. 
Our future work will include proving this conjecture, extending our ISS analysis to the case of different norms and further developing our results to analyze firing-rate neural networks. Additionally, a possible future research direction could be applying our result for time-critical optimization problems.
\section*{Acknowledgement}
The authors wish to thank Eduardo Sontag for stimulating conversations about contraction theory.

This work was in part supported by AFOSR project FA9550-21-1-0203 and NSF Graduate Research Fellowship under Grant No. 2139319.

Giovanni Russo wishes to acknowledge financial support by the European Union - Next Generation EU, under PRIN 2022 PNRR, Project “Control of Smart Microbial Communities for Wastewater Treatment”.

\appendices
\section{Contraction times with respect to distinct norms}\label{apx:proof_lemma}

First we recall the following~\cite[Lemma V.1]{VC-AG-AD-GR-FB:23a}
\begin{lem}[Inclusion between balls computed with respect to different norms]
\label{lem:ineq_balls_inclusions}
Given two norms $\norm{\cdot}_{\alpha}$ and $\norm{\cdot}_{\beta}$ on $\R^n$, for all $x \in \R^n$ and $r>0$, it holds
\begin{equation}
\label{ineq:balls:inclusions}
\ball{ }{x, r/\kab; \beta} \subseteq \ball{}{x, r; \alpha} \subseteq  \ball{}{x, r \kba; \beta}.
\end{equation}
\end{lem}
The following Lemma is inspired by~\cite[Theorem V.2]{VC-AG-AD-GR-FB:23a}. For completeness, we here provide a self-contained proof.
\begin{lem}[Contraction times with respect to distinct norms]
\label{lem:contraction_time}
Given $\norm{\cdot}_{\alpha}$ and $\norm{\cdot}_{\beta}$ norms on $\R^n$ with equivalence ratio $\prodeqnorm_{\alpha, \beta}$, consider system~\eqref{eq:dynamical_system} satisfying Assumptions~\ref{ass:local-strong_contractivity},~\ref{ass:eq_point} with $\norm{\cdot}_\loc = \norm{\cdot}_{\alpha}$. Then, for each $0<\rho<1$,
\begin{enumerate}
\item\label{fact:contraction-time} the $\rho$-contraction time is $t_\rho = \ln(\rho^{-1})/c$;
\item\label{fact:contraction-time-mixed} the $\rho$-contraction time with respect to the norm $\norm{\cdot}_{\beta}$ is $t_\rho^{\alpha,\beta} = \ln(\ratioab\, \rho^{-1})/c$.
\end{enumerate}
\end{lem}
\begin{proof}
Consider a trajectory $x(t)$ of system~\eqref{eq:dynamical_system} such that $\norm{x_0}_{\alpha}\leq{r}$. To prove~\ref{fact:contraction-time} we need to find the first time $t_\rho$ such that $\norm{x(t_\rho) - \xstar}_{\alpha}\leq{\rho r}$.
Clearly the worst-case time is achieved when $\norm{x_0 - \xstar}_{\alpha}={r}$. But $c$-strongly infinitesimal contractivity with respect to $\norm{\cdot}_{\alpha}$ implies $\norm{x(t) - \xstar}_{\alpha}\leq\e^{-ct}\norm{x_0 - \xstar}_{\alpha}$ and so $t_\rho$ is determined by the equality $\e^{-ct_\rho}r=\rho r$, from which item~\ref{fact:contraction-time} follows.

Regarding item~\ref{fact:contraction-time-mixed}, we need to find the first time $t_\rho^{\alpha,\beta}$ such that $\norm{x(t_\rho) - \xstar}_{\beta}\leq{\rho r}$. We note that
\begin{align*}
x_0\in \ball{\beta}{\xstar, r} \quad & \overset{\eqref{ineq:balls:inclusions}, \text{ $\supscr{2}{nd}$ inequality}}{\implies}\quad x_0\in \ball{\alpha}{\xstar, \kab r}, \\
x(t_\rho)\in \ball{\beta}{\xstar, \rho r} \quad & \overset{\eqref{ineq:balls:inclusions}, \text{ $\supscr{1}{st}$ inequality}}{\impliedby}\quad x(t_\rho)\in \ball{\alpha}{\xstar, \rho r/\kba}.
\end{align*}
Thus, the contraction time from $\ball{\beta}{\xstar, r}$ to $\ball{\beta}{\xstar, \rho r}$ is upper bounded by the contraction time from $\ball{\alpha}{\xstar, \kab r}$ to $\ball{\alpha}{\xstar, \rho r/\kba}$. Therefore, the contraction factor with respect to the $\norm{\cdot}_{\alpha}$ norm is $(\rho r/\kba) / (\kab r) = \rho / \ratioab$. Item~\ref{fact:contraction-time-mixed} then follows from item~\ref{fact:contraction-time}.
\end{proof}

\bibliographystyle{plainurl+isbn}
\bibliography{alias, Main, FB}
\end{document}

%% file: my_commands.tex
\newtheorem{definition}{Definition}
\newtheorem{thm}{Theorem}[section]
\newtheorem{lem}[thm]{Lemma}
\newtheorem{pro}[thm]{Proposition}

\newtheorem{conj}[thm]{Conjecture}
\newtheorem{remark}[thm]{Remark}

\newmdtheoremenv{assumptionbox}{Assumption}
\newenvironment{rem}{\begin{remark}}{\qed\end{remark}}
\newcommand{\bd}{\begin{definition}} 
\newcommand{\ed}{\end{definition}} 
\newcommand{\bp}{\begin{pro}} 
\newcommand{\ep}{\end{pro}} 
\newcommand{\bt}{\begin{thm}} 
\newcommand{\et}{\end{thm}}
\newcommand{\blm}{\begin{lem}} 
\newcommand{\elm}{\end{lem}}

\newcommand{\Z}{\mathbb{Z}}
\newcommand{\R}{\mathbb{R}}
\newcommand{\mcC}{\mathcal{C}}
\newcommand{\mcX}{\mathcal{X}}
\newcommand{\mcY}{\mathcal{Y}}
\newcommand{\mcU}{\mathcal{U}}
\newcommand{\mcS}{\mathcal{S}}

\newcommand{\indicator}[1]{\mathbf{1}_{#1}}
\newcommand{\realextended}{\overline{\R}}
\newcommand{\bi}{\begin{itemize}} 
\newcommand{\ei}{\end{itemize}}
\newcommand{\beq}{\begin{equation}} 
\newcommand{\eeq}{\end{equation}} 

\newcommand{\abs}[1]{\left|#1\right|}

\newcommand{\prox}[1]{\mathrm{prox}_{#1}}
\newcommand{\proj}{\mathbb{P}}
\newcommand{\norm}[1]{\|#1\|}

\newcommand{\WP}[2]{\left\llbracket {#1}, {#2}\right\rrbracket}

\newcommand{\lognorm}[2]{\mu_{#1}(#2)}
\newcommand{\wlognorm}[3]{\mu_{#1,#2}\bigl(#3\bigr)}

\newcommand{\subscr}[2]{#1_{\textup{#2}}}
\newcommand{\supscr}[2]{#1^{\textup{#2}}}
\newcommand{\setdef}[2]{\{#1 \; | \; #2\}}

\newcommand{\map}[3]{#1\colon #2 \rightarrow #3}

\DeclareMathOperator*{\argmin}{argmin}
\newcommand{\ceil}[1]{\lceil #1 \rceil}
\newcommand{\Bigceil}[1]{\left\lceil #1 \right\rceil}
\newcommand{\ds}{\displaystyle}

\newcommand{\xstar}{x^{\star}}

\newcommand{\e}{\mathrm{e}}

\newcommand{\relu}{\operatorname{ReLU}}
\newcommand{\sat}[2]{\operatorname{sat}_{#1}({#2})}

\newcommand{\0}{\mbox{\fontencoding{U}\fontfamily{bbold}\selectfont0}}

\newcommand{\linexp}[1]{\operatorname{lin-exp}({#1})}

\newcommand{\satur}{\operatorname{sat}}

\newcommand{\odeflowtx}[2]{\phi_{#1}\bigl({#2}\bigr)}

\newcommand{\radius}{r}

\newcommand{\tstar}{t^{\star}}
\newcommand{\tld}{\subscr{t}{c}}
\newcommand{\prodeqnorm}{k}
\newcommand{\ce}{\subscr{c}{exp}}
\newcommand{\cl}{\subscr{c}{lin}}
\newcommand{\loc}{\textup{L}}
\newcommand{\glo}{\textup{G}}

\newcommand{\kab}{k_\alpha^\beta}
\newcommand{\kba}{k_\beta^\alpha}
\newcommand{\ratioab}{k_{\alpha,\beta}}
\newcommand{\ratiolg}{k_{\loc,\glo}}

\newcommand{\glbnorm}[1]{\|#1\|_{\glo}}
\newcommand{\ball}[2]{B_{#1}{\bigl(#2\bigr)}}
\newcommand{\gwlsc}{GW-LS-C\xspace}
 